\documentclass[11pt]{amsart} %

\usepackage{float}

\usepackage{epsfig}

\usepackage{array,amsmath, enumerate,  url, psfrag}
\usepackage{amssymb, amsaddr, fullpage}
\usepackage{graphicx,subfigure}
\usepackage[usenames,dvipsnames]{pstricks}
\usepackage{color}

\restylefloat{figure}
 \makeatother

\newtheorem{thm}{Theorem}[section]
\newtheorem{lem}[thm]{Lemma}

\newtheorem{cor}[thm]{Corollary}

\newtheorem{example}[thm]{Example}

\title{Planar graphs without pairwise adjacent 3-,4-,5-, and 6-cycle are 4-choosable}

\author{Pongpat Sittitrai \hskip 0.8in Kittikorn Nakprasit}
\date{}

\begin{document}

\maketitle

\begin{center}{\bf Abstract}\end{center}
\indent\indent
Xu and Wu proved that if every $5$-cycle of a planar graph $G$ 
is not simultaneously adjacent to $3$-cycles and $4$-cycles, 
then $G$ is $4$-choosable. In this paper, we improve this result as follows. 
If $G$ is a planar graph without pairwise adjacent $3$-,$4$-,$5$-, and $6-$cycle, 
then $G$ is $4$-choosable.   

\section{Introduction}
\indent Every graph in this paper is finite, simple, and undirected. 
The concept of choosability was introduced by Vizing in 1976 \cite{Vizing} 
and by Erd\H os, Rubin, and Taylor in 1979 \cite{Erdos}, independently. 
A $k$-\emph{assignment} $L$ of a graph $G$ assigns a list $L(v)$ 
(a set of colors) with $|L(v)|= k$ to each vertex $v.$ 
A graph $G$ is $L$-\emph{colorable} if there is a proper coloring $f$ where $f(v)\in L(v).$   
If $G$ is $L$-colorable for any $k$-assignment $L,$ then we say  $G$ 
is $k$-\emph{choosable}. 

It is known that every planar graphs is $4$-colorable \cite{app1, app2}.  
Thomassen \cite{Tho} proved that every planar graph is $5$-choosable. 
Meanwhile, Voight \cite{Vo1} presented an example of non $4$-choosable planar graph. 
Additionally, Gutner \cite{Gut} showed that determining 
whether a given planar graph $4$-choosable is NP-hard. 
Since every planar graph without $3$-cycle always has a vertex of 
degree at most $3,$ it is $4$-choosable. 
More  conditions for  a planar graph to be $4$-choosable are investigated. 
It is shown that a planar graph is $4$-choosable if it has no  
$4$-cycles \cite{Lam2}, $5$-cycles \cite{Wang1}, $6$-cycles \cite{Fi}, 
$7$-cycles \cite{Far}, intersecting $3$-cycles \cite{Wang2}, 
intersecting $5$-cycles \cite{Hu},  or 
$3$-cycles adjacent to $4$-cycles \cite{Bo, Cheng}. 
Xu and Wu  \cite{Xu} proved that if every $5$-cycle of a planar graph $G$ 
is not simultaneously adjacent to $3$-cycles and $4$-cycles, 
then $G$ is $4$-choosable. In this paper, we improve this result as follows. 

\begin{thm}\label{main}
	If $G$ is a planar graph without pairwise adjacent $3$-,$4$-,$5$-, and $6-$cycle, 
	then $G$ is $4$-choosable.   
\end{thm}

\section{Preliminaries}
\indent 
First, we introduce some notations and definitions. 

Let $G$ be a plane graph. We use $V(G),E(G),$ and $F(G)$ for the vertex set, the edge set, 
and the face set respectively.  
We use $B(f)$ to denote a boundary of a face $f.$ 
A \emph{wheel}  $W_n$ is an $n$-vertex graph formed 
by connecting a single vertex (\emph{hub}) to all vertices (\emph{external vertices}) of an $(n-1)$-cycle.     
A $k$-vertex ($k^+$-vertex, $k^-$-vertex, respectively) is 
a vertex of degree $k$ (at least $k,$ at most $k,$ respectively). 
The same notations are applied to faces. 

A \emph{$(d_1,d_2,\dots,d_k)$-face} $f$ is a face of degree $k$ 
where  vertices on $f$ have degree $d_1,d_2,\dots,d_k$ in  a cyclic order. 
A \emph{$(d_1,d_2,\dots,d_k)$-vertex} $v$ is a vertex of degree $k$ 
where  faces incident to $v$ have degree $d_1,d_2,\dots,d_k$ in a cyclic order. 
Note that some face may appear more than one time in the order. 
A face is called  \emph{poor}, \emph{semi-rich}, and \emph{rich}, respectively   
if it is incident to no $5^+$-vertices, exactly one $5^+$-vertex, 
and at least two $5^+$-vertices, respectively. 
An \emph{extreme} face is a bounded face that shares a vertex with the unbounded face. 
An \emph{inner} face is a bounded face that is not an extreme face. 
A $(3,5,3,5^+)$-vertex $v$ is called a \emph{flaw $4$-vertex} 
if $v$ is incident to a poor $5$-face and all incident faces of $v$ are inner faces.  

%Unless stated otherwise, a face is a bounded face but is not an extreme face.  

We say $xy$ is a \emph{chord} in an embedding cycle $C$ if $x,y\in V(C)$ but $xy\in E(G)- E(C).$  
An \emph{internal chord} is a chord inside $C$
while  \emph{external chord} is a chord outside $C.$ 
A \emph{triangular chord} is a chord $e$ such that two edges in $C$ and $e$ form a $3$-cycle. 
A graph $C(m,n)$ is  obtained from  
a cycle $x_1x_2\ldots x_{m+n-2}$ with an internal chord $x_1x_m.$   
%For example, cycles $uvw$ and $vwxyz$ form $C(3,5).$ 
A graph $C(l,m,n)$ is  obtained from 
a cycle $x_1x_2\ldots x_{l+m+n-4}$ with internal chords $x_1x_l$ and $x_1x_{l+m-2}.$   
A graph $C(m,n,p,q)$ can be defined similarly. 
%For example, cycles $rsuv,$ $uvw,$ and $vwxyz$ form $C(4,3,5).$ 
We use $int(C)$ and $ext(C)$ to denote the graphs induced by  
vertices inside and outside a cycle $C$, respectively. 
The cycle $C$ is a {\em separating cycle} if $int(C)$ and $ext(C)$ are not empty.

%It is straightforward to see that if $f$ is a $5^-$-face, then $B(f)$ is a cycle. 

Let $L$ be a list assignment of $G$ and let $H$ be an induced subgraph of $G.$ 
Suppose $G - H$  has an $L$-coloring $\phi$ on $G-H$ where $L$ is restricted to $G-H.$ 
For a vertex $v \in H,$ let $L''(v)$ be a set of colors used on the neighbors of $v$ by $\phi.$   
We define a \emph{residual list assignment} $L'$ of $H$ by $L'(v) = L(v) - L''(v).$ 
One can see that if $G-H$ has an an $L$-coloring $\phi$ and $H$ has an $L'$-coloring, 
then $G$ has an $L$-coloring. 

%\begin{defn}\label{residual} 
%	Let $L$ be a list assignment of $G$ and let $G'$ be an induced subgraph of $G.$ 
% 	A list assignment $L'$ is a \emph{restriction of $L$} on $G'$ 
% 	if $L'(u) = L(u)$ for each vertex in $G'.$ 
%	Assume $G- G'$ has an $L$-coloring $\phi$.\\
%	\indent A \emph{residual list assignment} $L'$ of $G'$ is defined by    
%	$$L'(x)=L(x)-\bigcup_{ux\in E(G)}\{b\in L(x) : 
%	\phi(x)=b$$
%	for each $x \in V(G').$ \\
%\end{defn}

The following is a fact on list colorings that we use later.

\begin{lem}\cite{Erdos}\label{listcycle} Let $L$ be a $2$-assignment. 
	A cycle $C_n$ is $L$-colorable if and only if 
	$n$ is even or $L$ does not assign the same list to all vertices.  
\end{lem}

Let $\mathcal{A}$ denote the family of planar graphs 
without pairwise adjacent $3$-,$4$-,$5$-, and $6-$cycle. 

Next, we explore some properties of graphs in  $\mathcal{A}$ which are helpful in a proof of the main results.  

\begin{lem}\label{forbid1} 
	Every graph $G$ in  $\mathcal{A}$ does not contain each of the followings:\\  
	(1) $C(3,3,4),$\\ 
	(2) $C(3,3,5),$\\ 
	(3) $C(3,4,4^-),$ \\ 
	(4) $C(4,3,5).$\\ 
	(5) $W_5$ that shares exactly one edge with a $6^-$-cycle.  
\end{lem} 

\begin{proof}
	Let $C(l,m,n)$ be obtained from  a cycle $x_1x_2\ldots x_{l+m+n-4}$ 
	with internal chords $x_1x_l$ and $x_1x_{l+m-2}.$\\ 
	\indent (1) Suppose $G$ contains $C(3,3,4).$  
	Then we have four pairwise adjacent cycles $x_1x_2x_3,$ $x_1x_2x_3x_4,$ $x_1x_3x_4x_5x_6,$ and
	$x_1x_2x_3x_4x_5x_6,$ contrary to $G \in \mathcal{A}.$\\  
	\indent (2) Suppose $G$ contains $C(3,3,5).$  
	Then we have four pairwise adjacent cycles $x_1x_3x_4,$ $x_1x_2x_3x_4,$ $x_1x_4x_5x_6x_7,$ and
	$x_1x_3x_4x_5x_6x_7,$ contrary to $G \in \mathcal{A}.$\\
	\indent (3) Suppose $G$ contains $C(3,4,3).$  
	Then we have four pairwise adjacent cycles $x_1x_2x_3,$ $x_1x_3x_4x_5,$ $x_1x_2x_3x_4x_5,$ and
	$x_1x_2x_3x_4x_5x_6,$ contrary to $G \in \mathcal{A}.$\\    
	\indent Suppose $G$ contains $C(3,4,4).$  
	Then we have four pairwise adjacent cycles $x_1x_2x_3,$ $x_1x_3x_4x_5,$ $x_1x_2x_3x_4x_5,$ and
	$x_1x_3x_4x_5x_6x_7,$ contrary to $G \in \mathcal{A}.$\\
    \indent (4) Suppose $G$ contains $C(4,3,5).$  
	Then we have four pairwise adjacent cycles $x_1x_4x_5,$ $x_1x_2x_3x_4,$ $x_1x_2x_3x_4x_5,$ and
	$x_1x_4x_5x_6x_7x_8,$ contrary to $G \in \mathcal{A}.$\\    
	\indent (5) Let the hub of $W_5$ be $q$ and let external vertices be $r,s,u,$ and $v$ in a cyclic order.\\
	\indent Suppose there is a cycle $uvw.$ 
	Then we have four pairwise adjacent cycles $vwu,$ $vwuq,$ $vwusq,$ and
	$vwusqr,$ contrary to $G \in \mathcal{A}.$\\    
	\indent Suppose there is a cycle $uvwx.$ 
	Then we have four pairwise adjacent cycles $usq,$ $usqv,$ $usqrv,$ and
	$usqvwx,$ contrary to $G \in \mathcal{A}.$\\  
	\indent Suppose there is a cycle $uvwxy.$ 
	Then we have four pairwise adjacent cycles $uqv,$ $uqrv,$ $uqsrv,$ and
	$uqvwxy,$ contrary to $G \in \mathcal{A}.$\\  
	\indent Suppose there is a cycle $uvwxyz.$ 
	Then we have four pairwise adjacent cycles $uvq,$ $uvqs,$ $uvqrs,$ and
	$uvwxyz,$ contrary to $G \in \mathcal{A}.$\\  
\end{proof}

\begin{lem}\label{C(3,5)} 
	If $C$ is a $6$-cycle with a triangular chord, 
	then $C$ has exactly one chord. 
\end{lem}
\begin{proof}
	Let $C=tuvxyz$ with a chord $tv.$  Suppose to the contrary that $C$ has another chord $e.$ 
	By symmetry, it suffices to assume that $e=ux, uy, tx, ty,$ or $xz.$\\ 
	\indent If $e=ux,$ then we have four pairwise adjacent cycles $tuv, tuxv, tvxyz,$ and $tuvxyz,$ 
	contrary to $G \in \mathcal{A}.$\\ 
	\indent If $e=uy,$ then we have four pairwise adjacent cycles $tuv, uvxy, tvxyz,$ and $tuvxyz,$  
	contrary to $G \in \mathcal{A}.$\\ 
	\indent If $e=tx,$ then we have four pairwise adjacent cycles $tuv, tuvx, tvxyz,$ and $tuvxyz,$  
	contrary to $G \in \mathcal{A}.$\\ 
	\indent If $e=ty,$ then we have four pairwise adjacent cycles $tuv, tvxy, tvxyz,$ and $tuvxyz,$  
	contrary to $G \in \mathcal{A}.$\\ 
	\indent If $e=xz,$ then we have four pairwise adjacent cycles $tuv, tvxz, tvxyz,$ and $tuvxyz,$  
	contrary to $G \in \mathcal{A}.$\\
	\indent Thus $C$ has exactly one chord. 
\end{proof}

\section{Structure}  

To prove Theorem~\ref{main}, we prove a stronger result as follows. 

\begin{thm}\label{strong}
	If $G \in \mathcal{A}$ with a $4$-assignment $L,$ 
	then each precoloring of a $3$-cycle in $G$ can be extended to an $L$-coloring of $G$.
\end{thm} 

If $G$ does not contain a $3$-cycles, then $G$ is $4$-choosable as stated above.  
So we consider $(G, C_0)$ and a $4$-assignment $L$ where $C_0$ is a precolored $3$-cycle 
as a minimal counterexample to Theorem~\ref{strong}. Embed $G$ in the plane. 

\begin{lem}\label{separating}
	$G$ has no separating $3$-cycles.
\end{lem}

\begin{proof}
	Suppose to the contrary that there exists $G$ contains a separating $3$-cycle $C.$ 
	By symmetry, we assume $V(C_0) \subseteq V(C) \cup int(C).$ 
	By the minimality of $G,$ a precoloring of $C_0$ can be extended to $V(C) \cup int(C).$  
	After $C$ is colored, then again the coloring of $C$ can be extended to $ext(C)$. 
	Thus we have an $L$-coloring of $G$, a contradiction.
\end{proof}

So we may assume that a minimal counterexample $(G,C_0)$ 
has no separating $3$-cycles 
and $C_0$ is the boundary of the unbounded face $D$ of $G$ in the rest of this paper.

\begin{lem}\label{minimum}
	Each vertex in $int(C_0)$ has degree at least four.
%If $G$ is a minimal planar graph that is not DP-$4$-colorable, then each vertex that is not precolored in $G$ has degree at least $4$.
\end{lem}

\begin{proof}
	Suppose otherwise that there exists a $3^-$-vertex $v$ in $int(C_0)$. 
	By the minimality of $(G,C_0)$, $(G-v,C_0)$ has an $L$-coloring. 
	One can see that the residual list $L'(v)$ is not empty. 
	Thus we can color  $v$ and thus extend a coloring to $G,$ a contradiction.  
\end{proof}

\begin{lem}\label{config} For faces in $G,$ each of the followings holds.\\
%For each of bounded faces (including extreme faces), the followings hold.  
	(1) The boundary of a bounded $6^-$-face is a cycle.\\
	(2) If a  bounded $k_1$-face $f$ and a bounded $k_2$-face $g$ are adjacent 
	where $k_1+k_2 \leq 8,$ 
	then $B(f)\cup B(g) = C(k_1,k_2).$\\
	(3) If  a bounded $4$-face $f$ and a bounded $5$-face $g$ are adjacent, then $B(f)\cup B(g)$ is $C(4,5)$ 
	or a configuration as in Figure 1 where $tuy$ is $C_0.$\\
	(4) If bounded $5$-faces $f$ and $g$ are adjacent, then $B(f)\cup B(g)$ is $C(5,5)$ 
	or a configuration as in Figure 2. 
\end{lem} 

\begin{proof} 
	(1) One can observe that a boundary of a $5^-$-face is always a cycle. 
	Consider a bounded $6$-face $f.$ If $B(f)$ is not a cycle, then a boundary closed walk is in a form of 
	$uvwxywu.$ By Lemma~\ref{minimum}, $u$ or $x$ has degree at least $4.$   
	Consequently, $uvw$ or $xyw$ is a separating $3$-cycle, contrary to Lemma~\ref{separating}.\\ 
	\indent (2) It suffices to show that such $f$ and $g$ share exactly two vertices.  
	Let $B(f)=uvw$ and $B(g)=vwx.$ If $u=x,$ then $f$ or $g$ is  the unbounded face, a contradiction.\\ 
	\indent Let $B(f)=uvw$ and $B(g)=vwxy.$ If $u = x$ or $y,$   
	then $d(w) =2$ or $d(v)=2,$ contrary to Lemma~\ref{minimum}.\\
    \indent Let $B(f)=uvw$ and $B(g)=vwxyz.$ If $u = x$ or $z,$ 
	then $d(v) =2$ or $d(w)=2,$ contrary to Lemma~\ref{minimum}.
	If $u=y,$ then $vyz$ or $wxy$ is a separating $3$-cycle, contrary to Lemma~\ref{separating}.\\ 
	\indent Let $B(f)=stuv$ and $B(g)=uvwx.$ If $s=w,$ then $d(v)=2,$ contrary to Lemma~\ref{minimum}. 
	If $s=x,$ then $utx$ or $vwx$ is a separating $3$-cycle, contrary to Lemma~\ref{separating}. 
	The remaining cases are similar. \\ 
	\indent (3) Let $B(f)=stuv$ and $B(g)=uvwxy.$ 
	It suffices to show that $V(B(f)) \cap V(B(g)) = \{u,v\}$ or $\{u, v, x\}$ where $x=s$ or $t.$  
	If $t=w,$ then $svw$ or $uvw$ is a separating $3$-cycle, contrary to Lemma~\ref{separating}. 
	If $t=x,$ then $tuy$ is $C_0,$ otherwise $tuy$ is a separating cycle, contrary to Lemma~\ref{separating}. 
	The remaining cases are similar.\\ 
	\indent (4) Let $B(f)=rstuv$ and $B(g)=uvwxy.$ 
	It suffices to show that $V(B(f)) \cap V(B(g)) = \{u,v\}$ or $\{u, v, x=s\}.$    
	If $r=w,$ then $d(v)=2,$ contrary to Lemma~\ref{minimum}.
	If $B(f) \cap B(g)= \{ u,v,r=x\},$ 
	then $vwx, uvxy, uvwxy,$ and $stuvwx$ are four pairwise adjacent cycles, contrary to $G \in \mathcal{A}.$ 
	If $B(f) \cap B(g)= \{ u,v,r=x, s=y\},$ 
	then $rvx, rvuz, rvuts,$ and $rstuvs$ are four pairwise adjacent cycles, contrary to $G \in \mathcal{A}.$ 
	then $uts$ or $vwx$ is a separating $3$-cycle, contrary to Lemma~\ref{separating}.
	If $B(f) \cap B(g)= \{ u,v,r=y\},$ then $ruv$ is a separating $3$-cycle, contrary to Lemma~\ref{separating}.
	If $B(f) \cap B(g)= \{ u,v,s=w\},$ 
	then $rvw, tuvw, uvwxy,$ and $rwxyuv$ are four pairwise adjacent cycles, contrary to $G \in \mathcal{A}.$ 
	The remaining cases are similar.\\ 	 
\end{proof}

%\huge{dots with size 9.0 and edges with size 1.2}

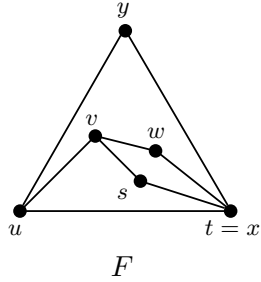
\begin{figure}[ht]\label{fig3}
\centering
\scalebox{1} % Change this value to rescale the drawing.
{
\begin{pspicture}(0,-1.8785938)(3.4871874,1.8785938)
\psdots[dotsize=0.18](1.663125,1.4289062)
\psdots[dotsize=0.18](2.063125,-0.17109375)
\psdots[dotsize=0.18](1.263125,0.02890625)
\psdots[dotsize=0.18](0.263125,-0.9710938)
\psdots[dotsize=0.18](3.063125,-0.9710938)
\usefont{T1}{ptm}{m}{n}
\rput(1.6245313,-1.7010938){$F$}
\usefont{T1}{ptm}{m}{n}
\rput(3.0867188,-1.1960938){\footnotesize $t=x$}
\psdots[dotsize=0.18](1.863125,-0.57109374)
\usefont{T1}{ptm}{m}{n}
\rput(0.19671875,-1.2160938){\footnotesize $u$}
\usefont{T1}{ptm}{m}{n}
\rput(2.0767188,0.06390625){\footnotesize $w$}
\psline[linewidth=0.024cm](1.263125,0.02890625)(1.863125,-0.57109374)
\psline[linewidth=0.024cm](1.263125,0.02890625)(2.063125,-0.17109375)
\psline[linewidth=0.024cm](1.263125,0.02890625)(0.263125,-0.9710938)
\psline[linewidth=0.024cm](1.703125,1.3689063)(3.023125,-0.89109373)
\psline[linewidth=0.024cm](1.603125,1.3689063)(0.303125,-0.89109373)
\psline[linewidth=0.024cm](0.343125,-0.9710938)(2.983125,-0.9710938)
\psline[linewidth=0.024cm](1.923125,-0.5910938)(2.983125,-0.93109375)
\psline[linewidth=0.024cm](2.123125,-0.21109375)(3.003125,-0.9110938)
\usefont{T1}{ptm}{m}{n}
\rput(1.2167188,0.24390624){\footnotesize $v$}
\usefont{T1}{ptm}{m}{n}
\rput(1.6367188,1.7039063){\footnotesize $y$}
\usefont{T1}{ptm}{m}{n}
\rput(1.6267188,-0.73609376){\footnotesize $s$}
\end{pspicture} 
}
\caption{A graph $F$ is formed by a $4$-face and a $5$-face with $tuy=C_0$}
\end{figure}

\begin{figure}[ht]\label{fig2}
\centering
\scalebox{1} % Change this value to rescale the drawing.
{
\begin{pspicture}(0,-1.9685937)(4.5871873,1.9685937)
\psdots[dotsize=0.18](2.223125,1.4789063)
\psdots[dotsize=0.18](2.223125,-0.12109375)
\psdots[dotsize=0.18](2.223125,-0.92109376)
\psdots[dotsize=0.18](0.423125,0.47890624)
\psdots[dotsize=0.18](4.023125,0.47890624)
\psline[linewidth=0.024cm](2.223125,-0.12109375)(2.223125,-0.92109376)
\usefont{T1}{ptm}{m}{n}
\rput(2.2145312,-1.7910937){$H$}
\usefont{T1}{ptm}{m}{n}
\rput(2.2667189,1.7939062){\footnotesize $s=x$}
\psline[linewidth=0.024cm](0.423125,0.47890624)(2.223125,1.4789063)
\psline[linewidth=0.024cm](2.223125,1.4789063)(4.023125,0.47890624)
\psline[linewidth=0.024cm](4.023125,0.47890624)(2.223125,-0.92109376)
\psline[linewidth=0.024cm](2.223125,-0.92109376)(0.423125,0.47890624)
\psdots[dotsize=0.18](1.623125,0.47890624)
\psdots[dotsize=0.18](2.823125,0.47890624)
\psline[linewidth=0.024cm](1.623125,0.47890624)(2.223125,1.4789063)
\psline[linewidth=0.024cm](2.223125,1.4789063)(2.823125,0.47890624)
\psline[linewidth=0.024cm](2.823125,0.47890624)(2.223125,-0.12109375)
\psline[linewidth=0.024cm](1.623125,0.47890624)(2.223125,-0.12109375)
\usefont{T1}{ptm}{m}{n}
\rput(0.17671876,0.61390626){\footnotesize $r$}
\usefont{T1}{ptm}{m}{n}
\rput(1.3667188,0.5939062){\footnotesize $t$}
\usefont{T1}{ptm}{m}{n}
\rput(3.0767188,0.5939062){\footnotesize $y$}
\usefont{T1}{ptm}{m}{n}
\rput(4.2767186,0.57390624){\footnotesize $w$}
\usefont{T1}{ptm}{m}{n}
\rput(2.4367187,-0.22609375){\footnotesize $u$}
\usefont{T1}{ptm}{m}{n}
\rput(2.1967187,-1.2260938){\footnotesize $v$}
\end{pspicture} 
}
\caption{A graph $H$ is formed by two adjacent $5$-faces but is not $C(5,5)$}
\end{figure}
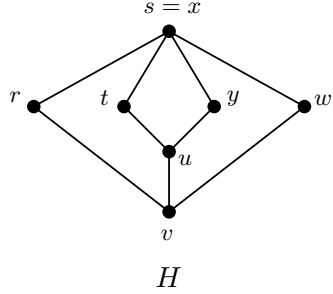

\begin{lem}\label{forbid2} 
	If a $k$-vertex $v$ is incident to bounded faces $f_1,\ldots,f_k$ 
	in a cyclic order and $d_i$ is a degree of a face $f_i$ for each $i\in\{1,\ldots,k\},$ 
	then each of the followings holds.\\
	(1) $(d_1,d_2,d_3) \neq (3,3,4).$\\
	(2) $(d_1,d_2,d_3) \neq (3,3,5).$\\ 
	(3) $(d_1,d_2,d_3) \neq (3,4,4^-).$\\ 
	(4) $(d_1,d_2,d_3) \neq (4,3,5).$\\ 
	(5) Let $H$ be $W_5$ such that  a hub and each two vertices of consecutive external vertices  
	form a boundary of an inner $3$-face. Then $H$ is not adjacent to a boundary of a $6^-$-face 
	other than these $3$-faces.    
\end{lem} 

\begin{proof}
	Let $F = B_1 \cup B_2 \cup B_3$ where $B_i$ denote $B(f_i).$\\ 
	\indent (1) Suppose $(d_1,d_2,d_3) = (3,3,4).$
	Let $B_1 = rsv,$ $B_2= vst,$ and $B_3=vtxy.$ 
	It follows from Lemma~\ref{config}~(2) that $V(B_1) \cap V(B_2) = \{s, v\}$ 
	and $V(B_2) \cap V(B_3) = \{t,v\}.$ 
	If $r=x,$  then $stx$ or $vxy$ is a separating $3$-cycle, contrary to Lemma~\ref{separating}.
	If $r=y,$ then $d(v)=3,$ contrary to Lemma~\ref{minimum}.  
	Thus $V(B_1) \cap V(B_3) =\{v\}.$ 
	Altogether we have $F=C(3,3,4),$ contrary to Lemma~\ref{forbid1}(1).\\
	\indent (2) Suppose $(d_1,d_2,d_3) = (3,3,5).$
	Let $B_1 = rsv,$ $B_2= vst,$ and $B_3=vtxyz.$ 
	It follows from Lemma~\ref{config}~(2) that 
	$V(B_1) \cap V(B_2) = \{s, v\}$ and $V(B_2) \cap V(B_3) = \{t,v\}.$    
	We have $C= stxyzv$ is a $6$-cycle with a triangular chord $tv.$ 
	If $r \in \{x,y,z\},$ then $C$ has another chord, contrary to Lemma~\ref{C(3,5)}. 
	Thus $V(B_1) \cap V(B_3) =\{v\}.$ 
	Altogether we have $F=C(3,3,5),$ contrary to Lemma~\ref{forbid1}(2).\\ 
	\indent (3) Suppose $(d_1,d_2,d_3) = (3,4,3).$
	Let $B_1 = rsv,$ $B_2= vstu,$ and $B_3=vuw.$ 
	It follows from Lemma~\ref{config}~(2) that $V(B_1) \cap V(B_2) = \{s, v\}$ 
	and $V(B_2) \cap V(B_3) = \{u,v\}.$    
	If $r=w,$ then $d(v)=3,$ contrary to Lemma~\ref{minimum}.  	
	Thus $V(B_1) \cap V(B_3) =\{v\}.$ 
	Altogether we have $F=C(3,4,3),$ contrary to Lemma~\ref{forbid1}(3).\\ 
	\indent Suppose $(d_1,d_2,d_3) = (3,4,4).$
	Let $B_1 = rsv,$ $B_2= vstu,$ and $B_3=uvxy.$ 
	It follows from Lemma~\ref{config}~(2) that $V(B_1) \cap V(B_2) = \{s, v\}$ 
	and $V(B_2) \cap V(B_3) = \{u,v\}.$    If $r=x,$ then $d(v)=3,$ contrary to Lemma~\ref{minimum}. 
	If $r=y,$ then $vuy$ is a separating $3$-cycle, contrary to Lemma~\ref{separating}. 
	Thus $V(B_1) \cap V(B_3) =\{v\}.$ 
	Altogether we have $F=C(3,4,4),$ contrary to Lemma~\ref{forbid1}(3).\\ 	
	\indent (4) Suppose $(d_1,d_2,d_3) = (4,3,5).$
	Let $B_1 = qrsv,$ $B_2= vst,$ and $B_3=vtxyz.$ 
	It follows from Lemma~\ref{config}~(2) that 
	$V(B_1) \cap V(B_2) = \{s, v\}$ and $V(B_2) \cap V(B_3) = \{t,v\}.$    
	We have $C= stxyzv$ is a $6$-cycle with a triangular chord $tv.$ 
	If $\{q,r\}$ and $\{x,y,z\}$ are not disjoint, then $C$ has another chord, contrary to Lemma~\ref{C(3,5)}. 
	Thus $V(B_1) \cap V(B_3) =\{v\}.$ 
	Altogether we have $F=C(4,3,5),$ contrary to Lemma~\ref{forbid1}(2).\\ 
	\indent (5) Let $v$ be a hub and let $w,x,y,z$ be external vertices of $H$ in the cyclic order. 
	Suppose to the contrary that $H$ is adjacent to a face $f$ with 
	$B(f)=wxq, wxqr, wxqrs,$ or  $wxqrst.$ 
	Now we have $ \{w,x\} \subseteq V(H) \cap V(B(f)).$ 
	By Lemma~\ref{forbid1}(5), $V(H) \cap V(B(f)) \neq \{w,x\}.$ 
	If $q=y,$ then $d(x) =3,$ contrary to Lemma~\ref{minimum}. 
	If $r=y,$ then $vwxqyz$ is a $6$-cycle with four triangular chords, contrary to Lemma~\ref{C(3,5)}. 
	If $s=y,$ then $vxw, vxwz, vxwzy,$ and $vxqryz$ are four pairwise adjacent cycles, 
	contrary to $G\in \mathcal{A}.$ 
	If $t=y,$ then $vxw, vxwz, vxwzy,$ and $vxqrsy$  are four pairwise adjacent cycles, 
	contrary to $G\in \mathcal{A}.$ 
	The remaining cases lead to similar contradictions. 
	Thus $f$ is not a $6^-$-face. 
\end{proof}

The following is rephrased from (\cite{Cheng}, Lemma 4).  
\begin{lem}\label{theta} 
	Let $C(m,n)$ in $int(C_0)$ be obtained from a cycle $C=x_1\ldots x_{m+n-2}$	
	with a chord  $x_1x_m$ and	$d(x_1) \leq 5.$ 
	If  $C$ has at most one additional chord $e$ and $e$ is not incident to $x_1,$ 
	then there exists $i\in \{2,\ldots,m+n-2\}$ with $d(x_i) \geq 5.$
\end{lem}

\begin{cor}\label{flaw}  
	If $v$ is a flaw vertex, then we have the followings.\\  
	(1) $v$ is incident to exactly one poor $5$-face.\\
	(2) Each $3$-face that is incident to $v$ is a semi-rich face.  
\end{cor}

\begin{proof}
	Let $v$ be incident to inner faces $f_1, f_2, f_3, f_4$ in a cyclic order where $f_1$ and $f_3$ are $3$-faces, 
	$f_2$ is a poor $5$-face, and $f_4$ is a $5^+$-face. By Lemma~\ref{config},  
	$B(f_1) \cup B(f_2)$ and $B(f_2) \cup B(f_3)$ are  $C(3,5).$ 
	It follows from  Lemma~\ref{theta}  that 
	some vertex in $B(f_1) \cup B(f_2)$ and in $B(f_2) \cup B(f_3)$ has degree at least $5.$ 
	Observe that some vertex in $B(f_1)$ and in $B(f_3)$ has degree at least $5$ since $f_2$ is a poor face.\\
	\indent (1) If $f_4$ is also a poor $5$-face, then $f_1$ is a poor face, contrary to the observation above.  
	\indent (2) By observation above, $f_1$ and $f_3$ are not poor $3$-faces. 
	Since $f_2$ is a poor face, we obtain that $f_1$ and $f_3$ are not rich faces. 
\end{proof}

\begin{lem}\label{strange}
	If  $H$ in Figure 2 is in $int(C_0)$ and contains a $5^-$-vertex $v,$  
	then there is another vertex of $H$ with degree at least $5$ in $G.$  
\end{lem}
\begin{proof} 
	First, we show that $H$ is an induced subgraph. 
	Suppose to the contrary that there is an edge $e$ joining vertices in $V(H)$ such that $e \notin E(H).$ 
	If $e=ty,$ then $sty$ or $tuy$ is a separating $3$-cycle. 
	If $e=ux,$ then $stu$  is a separating $3$-cycle. 
	If $e=sv,$ then $rsv$  is a separating $3$-cycle. 
	If $e=rw,$ then  $rvw$ is a separating $3$-cycle. 
	All consequences contradicts Lemma~\ref{separating}. Thus $H$ is an induced subgraph.\\
	\indent Suppose to the contrary that $d(v) \leq 5$ but each of remaining vertices has degree at most $4.$
	By minimality, $G-H$ has an $L$-coloring where $L$ is restricted to $G-H.$ 
	Consider a residual list assignment $L'$ on $H.$
	Since $L$ is a $4$-assignment, 
	we have $|L'(s)| = 4,$  $|L'(u)| \geq 3,$ and $|L'(v)|,|L'(r)|, |L'(t)|, |L'(y)|, |L'(w)| \geq 2.$    
	We begin by choosing a color $c$ from  $L'(u)$ such that 
	$|L'(y)-{c}|\geq 2.$  
	Then we choose colors of $v,r,w,t,s,$ and $y$ in this order, 
	we obtain an $L'$-coloring on $H$. 
	Thus we can extend an $L$-coloring to $G,$ a contradiction. 
\end{proof}

\begin{cor}\label{2faces}
	Let $v$ be a $k$-vertex in $int(C_0)$ 
	with consecutive incident faces $f_1,\ldots,f_k$ where $k\leq 5.$  
	If $f_1$ and $f_2$ are inner $5^-$-faces, 
	then there exists $w \in B(f_1) \cup B(f_2)$ such that $w \neq v$ and $d(w) \geq 5.$ 
\end{cor}

\begin{proof} 
		The result follows directly from Lemmas ~\ref{config}, \ref{theta}, and \ref{strange}. 
\end{proof}

\begin{cor}\label{5vertex3rich}
	If $v$ is a $5$-vertex in which each incident face is an inner $5^-$-face, 
	then $v$ is incident to at least three rich faces. 
\end{cor}

\begin{proof}
	Suppose to the contrary that $v$ is incident to at most two rich faces. 
	Consequently, $v$ is incident to two adjacent semi-rich faces $f$ and $g.$ 
	But $d(f) \leq 5$ and $d(g) \leq 5$ while all vertices in $B(f) \cup B(g)$ except $v$ has degree $4.$ 
	This contradicts Corollary~\ref{2faces}. 
\end{proof}

\begin{lem}\label{faces} 
	Let $C(l_1,\ldots,l_k)$ in $int(C_0)$ be obtained from a cycle $C=x_1 \ldots x_m$ 
	with  $k$ internal chords sharing a common endpoint $x_1.$ 
	Suppose $x_1$ is not incident to other chords 
	while $x_2$ or $x_m$ is not incident to any chord.   
	If $d(x_1) \leq k+3,$ then there exists $i \in \{2,3,\ldots,m\}$ such that $d(x_i) \geq 5.$  
\end{lem}
\begin{proof} 
	By symmetry, we assume $x_m$ is not an endpoint of any chord in $C.$ 
	Suppose to the contrary that $d(x_i) \leq 4$ for each $i =2,3,\ldots,m.$ 
	By the minimality of $G,$ 
	the subgraph $G-\{x_1,\ldots, x_m\}$ has an $L$-coloring 
	where $L$ is restricted to $G-\{x_1,\ldots, x_m\}$. 
	Consider a residual list assignment $L'$ on $x_1,\ldots, x_m.$
	Since $L$ is a $4$-assignment,  
	we have $|L'(x_1)|\geq 3$ and $|L'(v)| \geq 3$ for each $v\in V(C)$ with an edge $x_1v$   
	and $|L'(x_i)| \geq 2$ for each of the remaining vertices $x_i$ in $V(C).$  
	Since $x_m$ is not an endpoint of a chord in $C,$ we can choose a color $c$ from  $L'(x_1)$ such that 
	$|L'(x_{m})-{c}|\geq 2.$  
	By choosing colors of $x_2$, $x_3,\dots, x_{m}$ in this order, 
	we obtain an $L'$-coloring on $G'$. 
	Thus we can extend an $L$-coloring to $G,$ a contradiction. 
\end{proof}

\begin{cor}\label{C(5,3,5,3)} 
	Let $v$ be a $6$-vertex with consecutive incident faces $f_1,\ldots,f_6$ 
	and let $F = B_1 \cup B_2 \cup B_3\cup B_4$ where $B_i$ denote $B(f_i).$ 
	If $f_1\ldots f_4$ are bounded faces and $(d(f_1),d(f_2),d(f_3),d(f_4))= (5,3,5,3),$ 
	then there exists $w \in V(F)-\{v\}$ with $d(w) \geq 5.$
\end{cor}

\begin{proof}  
	By Lemma~\ref{faces}, it suffices to show that $F = C(5,3,5,3).$ 
	Let cycles $B_1= vqrst, B_2=vtu,$ $B_3=vuwxy,$ and $B_4=vyz.$ 
	Using Lemma~\ref{config}, we have that $V(B_1) \cap V(B_2)=\{v,t\},$ 
	$V(B_2) \cap V(B_3)=\{v,u\},$ and  $V(B_3) \cap V(B_4)=\{v,y\}.$ 
	It suffices to show that $V(B_1) \cap V(B_3)=\{v\}$ $=V(B_4) \cap (V(B_1) \cup V(B_2)).$\\
	\indent Suppose to the contrary that $V(B_1) \cap V(B_3)\neq \{v\}.$ 
	Consider a $6$-cycle $vtuwxy$ with a triangular chord $uv.$ 
	If $s=u,w,x,$ or $y,$ then $vtuwxy$ has another chord, contrary to Lemma~\ref{C(3,5)}. 
	Thus $s \notin V(B_1) \cap V(B_3).$ Similarly each of $q, w,$ and $y$ is not in $V(B_1) \cap V(B_3).$ 
	The only remaining possibility is that $r=x.$ Suppose this holds. 
	Then $vyz, vyxq, vyxwu,$ and $vyrstu$ are four pairwise adjacent cycle, contrary to $G \in \mathcal{A}.$ 
	Thus $V(B_1) \cap V(B_3)=\{v\}$ which implies $B_1 \cup B_2 \cup B_3 = C(5,3,5).$ 
	As a consequence, we have $vqrstu$ and $vtuwxy$ are $6$-cycles with a triangular chord.\\ 
	\indent 	If there is a vertex $b \in V(B_4) \cap (V(B_1) \cup V(B_2))$ such that $b \neq v,$ 
	then $vqrstu$ or $vtuwxy$ has another chord, contrary to Lemma~\ref{C(3,5)}. 
	This completes the proof. 
\end{proof}

\begin{cor}\label{W5}  
	Let $v$ be a $4$-vertex incident to four inner $3$-faces. 
	If all four neighbors of  $v$ are $5^-$-vertices, 
	then  at least three of them are $5$-vertices. 
\end{cor}
\begin{proof} 
	Let $w, x, y, z$ be  neighbor of $v$ in a cyclic order. 
	Let cycles $B_1=vwx$ and $B_2=vxy.$ 
	Note that $w$ and $y$ are not adjacent, otherwise $vwy$ is a separating $3$-cycle, 
	contrary to Lemma~\ref{separating}.
	Similarly, $x$ and $z$ are not adjacent.\\ 
	\indent Suppose to the contrary that there are at least two $4$-vertices among $w,x,y,$ and $z.$ 
	If those two $4$-vertices are not adjacent, say $w$ and $y,$ 
	then $B_1 \cup B_2$ contradicts Lemma~\ref{theta}. 
	Thus we assume that $w$ and $x$ are $4$-vertices.\\
	\indent Let $H$ be the graph induced by $v$ and its neighbors. 
	By minimality of $G,$ the graph $G -H$ has an $L$-coloring 
	where $L$ is restricted to $G -H.$ 
	Consider a residual list assignment $L'$ on $H.$ 
	Since $L$ is a $4$-assignment, 	we have $|L'(y)|, |L'(z)|\geq 2,$ 
	$|L'(w)|, |L'(x)|\geq 3,$ and $|L'(v)|= 4.$ 
	It suffices to assume that equalities holds for these list sizes. 
	We aim to show that $H$ has an $L'$-coloring,  
	and thus an $L$-coloring can be extended to $G,$ a contradiction.\\  
	\indent CASE 1. There is a color $t$ in $L'(v)- (L'(y)\cup L'(z))$.\\  
	We begin by choosing  $t$ for  $v.$ 
	Each of the residual lists of $w,x,y,z$ now has sizes at least $2.$ 	
	By Lemma~\ref{listcycle}, an even cycle is $2$-choosable, thus $H$ has an $L'$-coloring.\\  
	\indent CASE 2. $L'(v)- (L'(y)\cup L'(z))=\emptyset.$\\   
	This implies $L'(y)\cap L'(z)=\emptyset.$ 
	Choose $t\in L'(v) -L'(w)$ for $v.$ 
	If $t\in L'(y),$ then $t \notin L'(z)$ and we can color  $y,x,z,$ and $w$  in this order, 
	otherwise we can color  $z,y,x,$ and $w$  in this order.
	Thus $H$ has an $L'$-coloring. 
	This contradiction completes the proof. 
\end{proof}

\section{Proof of Theorem~\ref{strong}} 

\indent 
%Embed a minimal counterexample graph $G$ in the plane. 
Let the initial charge of a vertex $u$ in $G$ be $\mu(u)=2d(u)-6$ 
and the initial charge of a face $f$ in $G$ be $\mu(f)=d(f)-6$. 
Then by Euler's formula $|V(G)|-|E(G)|+|F(G)|=2$ and by the Handshaking lemma, we have
$$\displaystyle\sum_{u\in V(G)}\mu(u)+\displaystyle\sum_{f\in F(G)}\mu(f)=-12.$$ 
\indent Now we design the discharging rule transferring charge 
from one element to another to provide a new charge $\mu^*(x)$ 
for all $x\in V(G)\cup F(G).$ The total of new charges remains $-12$.  
If the final charge  $\mu^*(x)\geq 0$ for all $x\in V(G)\cup F(G)$, 
then we get a contradiction and complete the proof.\\ 

Before we establish a discharging rule, some definitions are required. 
% A cluster of three $3$-faces is  called a \emph{trio} 
% if it is isomorphic to a graph 
% consist of a vertex set of five elements, namely $\{u, v, w, x, y\},$ 
% and an edge set $\{xy, xu, xv, yv, yw, uv, vw\}.$ 
A graph $C(3,3,3)$  in $int(C_0$ is called a \emph{trio} .  
A vertex that is not in any trio is called a \emph{good} vertex. 
We call a vertex $v$ incident to a face $f$ in a trio $T$ 
a \emph{bad} (\emph{worse}, \emph{worst}, respectively) vertex of $f$ if $v$ is incident to exactly one 
(two, three, respectively) $3$-face(s) in $T.$ 
We call a face $f$ in a trio $T$ a \emph{bad} (\emph{worse},\emph{worst}, respectively) face  
of a vertex $v$ if $v$ is a bad (worse, worst, respectively) vertex of  $f$ in $T.$  
For our purpose, we regard an external vertex of  $W_5$ 
as a worse vertex of its incident $3$-faces in $W_5.$

\indent Let $w(v \rightarrow f)$ be the charge transferred from 
a vertex $v$ to an incident face $f.$ 
From now on, a vertex $v$ is in $int(C_0)$ unless stated otherwise. 
The discharging rules are as follows.\\
\textbf{(R1)} Let $f$ be an inner $3$-face that is not adjacent to another $3$-face.\\
\indent \textbf{(R1.1)} For a $4$-vertex $v$,\\
\indent \indent 
$
 w(v \rightarrow f) =
  \begin{cases}
   \frac{9}{10},        & \text{if } v \text{ is flaw},\\
   1,        & \text{otherwise.}
  \end{cases}
$\\
\indent \textbf{(R1.2)} For a $5^+$-vertex $v$,\\ 
\indent \indent 
$
 w(v \rightarrow f) =
  \begin{cases}

     \frac{6}{5},       & \text{if } f  \text{ is a } (4,4,5^+)\text{-face},\\
   1,        & \text{otherwise.}
  \end{cases}
$\\
\textbf{(R2)} Let $f$ be an inner $3$-face that is adjacent to another $3$-face.\\
\indent \textbf{(R2.1)} For a $4$-vertex $v$,\\ 
\indent \indent 
$
 w(v \rightarrow f) =
  \begin{cases}
   \frac{1}{2},        & \text{if } v \text{ is a hub of } W_5, \\
   1,        & \text{if } f \text{ is a good, bad, or worse face of } v,\\
   \frac{2}{3},        & \text{if } f \text{ is a worst face of } v.
  \end{cases}
$\\
\indent \textbf{(R2.2)} For a $5$-vertex $v$,\\
\indent \indent 
$
 w(v \rightarrow f) =
  \begin{cases}
   1,        & \text{if } f \text{ is a good or worst face of } v,\\
      \frac{5}{4},        & \text{if } f \text{ is a worse face of } v,\\
   \frac{3}{2},        & \text{if } f \text{ is a bad face of } v.
  \end{cases}
$\\
\indent \textbf{(R2.3)} For a $6^+$-vertex $v$,\\
\indent \indent 
$
 w(v \rightarrow f) =
  \begin{cases}
   1,        & \text{if } f \text{ is a good or worst face of } v,\\
   \frac{3}{2},        & \text{if } f \text{ is a bad or worse face of } v.
  \end{cases}
$\\
\textbf{(R3)} Let $f$ be an inner $4$-face.\\
\indent \textbf{(R3.1)} For a $4$-vertex $v$, let $w(v \rightarrow f) = \frac{1}{3}$.\\
\indent \textbf{(R3.2)} For a $5^+$-vertex $v$,\\
\indent \indent 
$
w(v \rightarrow f) =
  \begin{cases}
   1,        & \text{if } f \text{ is a } (4,4,4,5^+)\text{-face},\\
   \frac{2}{3},   & \text{if } f  \text{ is rich. } 
  \end{cases}
$\\
\textbf{(R4)} Let $f$ be an inner $5$-face.\\
\indent \textbf{(R4.1)} For a $4$-vertex $v$,\\ 
\indent \indent
$
 w(v \rightarrow f) =
  \begin{cases}
   \frac{1}{5},        & \text{if } v  \text{ is  flaw and } f \text{ is a poor } 5\text{-face},\\
       \frac{1}{3},        & \text{if } v  \text{ is  incident to at most one } 3\text{-face,}\\
     0,        & \text{otherwise.}
  \end{cases}
$\\
\indent \textbf{(R4.2)} For a $5^+$-vertex $v$,\\
\indent \indent 
$
 w(v \rightarrow f) =
  \begin{cases}
   1,        & \text{if } f \text{ is a } (4,4,4,4,5^+)\text{-face adjacent to five  } 3\text{-faces},\\
   \frac{2}{3},        & \text{if } f \text{ is a } (4,4,4,4,5^+)\text{-face  adjacent to at least one  } 4^+\text{-face other than } f,\\
     \frac{1}{t},        & \text{if } f \text{ is a rich face  with } t \text{ incident } 5^+\text{-vertices.}\\
  \end{cases}
$\\
\textbf{(R5)} Let $f$ be an inner $3$-face incident to a $4$-vertex $v$ 
that is incident to four $3$-faces. 
If $f$ is adjacent to  a $7^+$-face $g,$ we let $w(g \rightarrow f) = \frac{1}{8}$. \\
\textbf{(R6)} The unbounded face $D$ gets $\mu(v)$ from each incident vertex.\\
\textbf{(R7)} Let $f$ be an extreme face.\\
\indent \indent 
$
 w(x \rightarrow f) =
  \begin{cases}
  \frac{5}{2},        & \text{if } f \text{ is a 3-face that  } B(f) \text{ shares an edge with } C_0 \text{ and } x = D,\\
   2,        & \text{if } f \text{ is a 4- or 5-face and } x = D,\\
    2,        & \text{if } f \text{ is  a 3-face that  } B(f) \text{ shares exactly one vertex with }  C_0  
\text{ and } x = D,\\
             \frac{1}{2},        & \text{if } f \text{ is a 3-face and }  x \text{ is a vertex  in } int(C_0),\\ 
       0,        & \text{otherwise.}\\ 
  \end{cases}
$\\

\textbf{(R8)} After (R1) to (R7),  
redistribute the total of charges of $3$-faces in 
the same cluster of at least three adjacent inner $3$-faces (trio or $W_5$) equally among 
its $3$-faces.

\indent It remains to show that resulting $\mu^*(x)\geq 0$ 
for all $x\in V(G)\cup F(G)$. 
Let $v$ be a $k$-vertex incident to faces $f_1,\ldots,f_k$ in a cyclic order.  
By (R6), we only consider $v$ in $int(C_0).$ 
Consider the following cases. 
\begin{itemize} 
\item[(1)] $v$ is a $4$-vertex.
	\begin{itemize}
	\item[(1.1)] A vertex $v$ is incident to a $3$-face that is adjacent to another $3$-face.
		\begin{itemize} 
		\item[(1.1.1)] $v$ is incident to at least two consecutive $3$-faces.\\ 
		If $v$ is incident to  four $3$-faces, 
		then $v$ is a hub of $W_5$. 
		Thus $\mu^*(v)\geq \mu(v) - 4\times \frac{1}{2} = 0$ by (R2.1) and (R7). 
		If $v$ is incident to exactly three $3$-faces, say $f_1,f_2,$ and $f_3,$ 
		then $f_4$ is a $6^+$-face by Lemma~\ref{forbid2}(1) , (2). 
		Thus $\mu^*(v)\geq \mu(v) - 3\times \frac{2}{3} = 0$ by (R2.1) and (R7). 
		If $v$ is incident to exactly two $3$-faces, say $f_1$ and $f_2,$ 
		then $f_3$ and $f_4$ are $6^+$-faces by Lemma~\ref{forbid2}(1), (2). 
		Thus $\mu^*(v)\geq \mu(v) - 2\times 1 = 0$ by (R2.1) and (R7).
		\item[(1.1.2)]  $v$ has no adjacent incident $3$-faces.\\
		Let $f_1$ be a $3$-face adjacent to another $3$-cycle. 
		If follows from Lemma~\ref{forbid2}(1) and (2) that $f_2$ and $f_4$ are $6^+$-faces. 
		Then $w(v \rightarrow f_1) \leq  1$ by (R2.1) and (R7), 
		and $w(v \rightarrow f_3) \leq  1$ by (R2.1), (R3.1), (R4.1), and (R7).  
		Thus $\mu^*(v)\geq \mu(v) - 2\times 1 = 0.$
		\end{itemize}
	\item[(1.2)] $v$ is not incident to a $3$-face that is adjacent to another $3$-face 
	and $v$ is adjacent to at most one $3$-face.\\
	Using  the fact that $w(v \rightarrow f_i) \leq  1$ for a $3$-face $f_i$ by (R1.1) and (R7), 
	and $w(v \rightarrow f_i) \leq  \frac{1}{3}$ for each $4^+$-face $f_i$ by (R3.1), (R4.1), and (R7), 
	we obtain that  $\mu^*(v)\geq \mu(v) - 1 - 3\times \frac{1}{3} = 0.$
	\item[(1.3)] $v$ is not incident to a $3$-face that is adjacent to another $3$-face 
	and $v$ is adjacent to at least two $3$-faces.\\
	Consequently, $v$ is incident to exactly two $3$-faces, say $f_1$ and $f_3.$ 
	It follows from Lemma~\ref{forbid2}(3) that $f_2$ and $f_4$ are $5^+$-faces. 
	If $f_2$ is an extreme face, then  $w(v \rightarrow f_2) =0$ by (R7), 
	$w(v \rightarrow f_4)  =  0$ by (R4.1) and (R7), and 
	$w(v \rightarrow f_i)  \leq  1$ for $i=1$ and $3$ by (R1.1) and (R7). 
	Thus $\mu^*(v)\geq \mu(v) - 2\times 1 = 0.$ 
	Now assume that all incident faces of $v$ are inner faces. 
	If $v$ is flaw, then $v$ is incident to exactly one poor $5$-face, 
	say $f_2$ by Corollary~\ref{flaw}(1). 
	$w(v \rightarrow f_i)  =  \frac{9}{10}$ for $i=1$ and $3$ by (R1.1), 
	$w(v \rightarrow f_2)  \leq  \frac{1}{5}$ and  	$w(v \rightarrow f_4) =0$  
	by (R4.1) and (R7). 
	Thus $\mu^*(v) \geq \mu(v) - 2\times\frac{9}{10}-\frac{1}{5} = 0.$
	If $v$ is not flaw, then $w(v \rightarrow f_i)  =  1$ for $i=1$ and $3$ by (R1.1) 
	and 	$w(v \rightarrow f_i)  =0$ for $i=2$ and $4$  by (R4.1). 
	Thus $\mu^*(v) = \mu(v) - 2\times 1= 0.$
	\end{itemize}

\item[(2)] A  $5$-vertex $v$ is  incident to a $3$-face that is adjacent to another $3$-face.
	\begin{itemize} 
	\item[(2.1)]  $v$ has at least two consecutive incident $3$-faces.\\ 
	If $v$ is incident to four $3$-faces say $f_1,f_2,f_3,$ and $f_4$,  
	then one can see that $B(f_1)\cup B(f_2) \cup B(f_3) \cup B(f_4) = C(3,3,3,3).$ 
	But $C(3,3,3,3)$ contains four pairwise adjacent cycles that contradict $G \in \mathcal{A}.$
	Thus $v$ is incident to at most three consecutive $3$-faces.\\
	If $v$ incident to  consecutive three $3$-faces say $f_1,f_2,$ and $f_3$,  
	then $f_4$ and $f_5$ are $6^+$-faces by Lemma~\ref{forbid2}(1) and (2). 
	Thus $\mu^*(v) = \mu(v) - 3\times 1 > 0$ by (R2.2) and (R7).\\
	If $v$ incident to exactly two  consecutive $3$-faces say $f_1$ and $f_2$,  
	then $f_3$ and $f_5$ are $6^+$-faces by Lemma~\ref{forbid2}(1) and (2). 
	Consequently, $w(v \rightarrow f_i)  \leq  \frac{5}{4}$  for $i=1$ and $2,$  
	and $w(v \rightarrow f_4)  \leq  \frac{3}{2}$  by (R1.2), (R3.2), (R4.2), and (R7). 
	Thus $\mu^*(v) \geq \mu(v) -2\times\frac{5}{4}- \frac{3}{2} = 0.$  
	\item[(2.2)]  $v$ has no adjacent incident $3$-faces.\\ 
	Let $f_1$ be a $3$-face  adjacent to another $3$-face. 
	If follows from Lemma~\ref{forbid2}(1) and (2) that $f_2$ and $f_5$ are $6^+$-faces. 
	By (R2.2) and (R7),  $w(v \rightarrow f_1)  \leq  \frac{3}{2}.$ 
	If neither $f_3$ nor $f_4$ are $3$-faces, 
	then $w(v \rightarrow f_i)  \leq  1$  for $i=3$ and $4$ by (R3.2), (R4.2), and (R7). 
	Thus $\mu^*(v) \geq \mu(v) -\frac{3}{2} -2\times 1 > 0.$\\  
	Now assume that $f_3$ is a $3$-face. 
	By the condition of (2.2), $f_4$ is a $4^+$-face 
	which implies $w(v \rightarrow f_4)  \leq  1$   by (R3.2), (R4.2), and (R7).   
	If $f_3$ is adjacent to another $3$-face, then $f_4$ is a $6^+$-face by Lemma~\ref{forbid2}(1) and (2). 
	Moreover, $w(v \rightarrow f_3)  \leq  \frac{3}{2}$ by (R2.2) and (R7). 
	Thus $\mu^*(v) \geq \mu(v) -2\times\frac{3}{2} > 0.$  
	If $f_3$ is not adjacent to another $3$-face, then $w(v \rightarrow f_3)  \leq  1$ by (R2.2) and (R7). 
	Thus $\mu^*(v) \geq \mu(v) -\frac{3}{2}-2\times 1> 0.$  
	\end{itemize}
\item[(3)] A $5$-vertex $v$ is not incident to a $3$-face that is adjacent to another $3$-face 
	and $v$ is incident to at least one $6^+$-face.\\
	Consequently, $v$ is incident to at most two $3$-faces. 
	\begin{itemize}
	\item[(3.1)] $v$ is incident to at least two $6^+$-faces.\\  
	Recall that $w(v \rightarrow f_i)  \leq  \frac{6}{5}$ for each $3$-face $f_i$ by (R1.2) and (R7), 
	and $w(v \rightarrow f_i)  \leq  1$ for each $k$-face $f_i$ where $k=4,5$ by (R3.2), (R4.2), and (R7). 
	If $v$  is incident to $t$ $3$-faces, 
	then there are at most $3-t$ faces $f$ with $d(f)=4$ or $5.$ 
	Thus $\mu^*(v) \geq \mu(v) -t\times \frac{6}{5}- (3-t)\times 1> 0$ by $t\leq 3.$   	
	\item[(3.2)] $v$ is incident to exactly one $6^+$-face and incident to at most one $3$-face.\\ 
	If $v$ has no incident $3$-faces, 
	then $v$ has all incident faces $f$ except one $6^+$-face has $d(f)=4$ or $5.$ 
	Thus $\mu^*(v) \geq \mu(v) -4\times 1= 0$  by (R3.2), (R4.2), and (R7).\\ 
	Assume $v$ is incident to exactly one $3$-face, say $f_1.$ 	
	By Lemma~\ref{forbid2}(3), $v$ is not a $(3,4,4,4,6^+)$- or a $(3,4,4,6^+,4)$-face.  
	Consequently, $v$ has at least one incident $5$-face $f_j.$ 
	Moreover, $f_j$ is adjacent to at least one $4^+$-face. 
	We have $w(v \rightarrow f_1)  \leq  \frac{6}{5}$ by (R1.2) and (R7), 
	$w(v \rightarrow f_j)  \leq  \frac{2}{3}$ by (R4.2) and (R7), 
	and $w(v \rightarrow f_i)  \leq  1$ for each remaining $k$-face $f_i$ 
	where $k=4,5$ by (R3.2), (R4.2), and (R7). 
	Thus $\mu^*(v) \geq \mu(v) -\frac{6}{5}-\frac{2}{3}-2\times1   > 0.$ 
	\item[(3.3)] $v$ is incident to exactly one $6^+$-face and incident to exactly two $3$-faces, 
	say $f_1$ and $f_3.$ \\
	It follows from Lemma~\ref{forbid2}(3) and (4) that  $v$ is  either 
	a $(3,5,3,5,6^+)$-, $(3,5,5,3,6^+)$- or  $(3,5,4,3,6^+)$-vertex.\\ 
	Assume $v$ is  a $(3,5,3,5,6^+)$- or  $(3,5,5,3,6^+)$-vertex. 
	Applying Corollary~\ref{2faces} to $B(f_2) \cup B(f_3),$  
	$v$ has an incident $5$-face $f_j$ which is rich or extreme. 
	Recall that $w(v \rightarrow f_i)  \leq  \frac{6}{5}$ for each $3$-face $f_i$ by (R1.2) and (R7), 
	$w(v \rightarrow f_j)  \leq  \frac{1}{2}$ by (R4.2) and (R7), 
	and $w(v \rightarrow f_i)  \leq  1$ for the remaining $5$-face $f_i$ by (R4.2) and (R7). 
	Thus $\mu^*(v) \geq \mu(v) -2\times\frac{6}{5} -\frac{1}{2}-1 > 0.$\\
	Assume  $v$ is a $(3,5,4,3,6^+)$-vertex. 
	Applying Corollary~\ref{2faces} to $B(f_1) \cup B(f_2),$  
	we obtain that $f_1$ or $f_2$ is rich or extreme. 
	In the former case, $w(v \rightarrow f_1)  \leq  1$  by (R1.2) and (R7), 
	and $w(v \rightarrow f_2)  \leq  \frac{2}{3}$ by (R4.2) and (R7). 
	In the latter case, $w(v \rightarrow f_1)  \leq  \frac{6}{5}$  by (R1.2) and (R7), 
	and $w(v \rightarrow f_2)  \leq  \frac{1}{2}$ by (R4.2) and (R7). 
	Combining with $w(v \rightarrow f_3)  \leq  1$ by (R3.2) and (R7) and 
	$w(v \rightarrow f_4)  \leq  \frac{6}{5}$  by (R1.2) and (R7), 
	we have  $\mu^*(v) \geq \mu(v)  - 2\times1 - \frac{2}{3}-\frac{6}{5} > 0$ or 
	$\mu^*(v) \geq \mu(v) -2\times\frac{6}{5} - \frac{1}{2}- 1  >0 .$  
	\end{itemize} 

\item[(4)] A $5$-vertex $v$ is not incident to a $3$-face that is adjacent to another $3$-face 
	and $v$ is not incident to a $6^+$-face.\\	
	Consequently, $v$ is incident to at most two $3$-faces. 
	Using Corollary~\ref{5vertex3rich}, 
	we have that  $v$ has at least three incident faces that are rich or extreme.  
	\begin{itemize} 
	\item[(4.1)] $v$ has no incident $3$-faces.\\  
	If $f$ has an extreme face $f_i,$ 
	then  $w(v \rightarrow f_i)  =0$ by (R7) 
	and $w(v \rightarrow f_i)  \leq  1$ for each remaining $f_i$ by (R3.2), (R4.2), and (R7).  
	Thus $\mu^*(v) \geq \mu(v) -4\times 1 =0 .$\\ 
	If $f$ has $t$ rich faces, 
	then $\mu^*(v) \geq \mu(v) -t\times \frac{2}{3} -(5-t)\times 1 \geq 0 $ 
	by (R3.2), (R4.2), (R7), and $t\geq 3.$
	\item[(4.2)] $v$ is incident to exactly one $3$-face, say $f_1.$\\ 
	It follows from  Lemma~\ref{forbid2}(3) that $v$ has at most two incident $4$-faces. 
		\begin{itemize}
			\item[(4.2.1)] $v$ has no incident  $4$-faces.\\
			We have that $w(v \rightarrow f_1)  \leq  \frac{6}{5}$  by (R1.2) and (R7) 
			and $w(v \rightarrow f_i)  \leq  \frac{2}{3}$ for each $5$-face $f_i$ by (R4.2)  and  (R7).  
			Thus $\mu^*(v) \geq \mu(v) -\frac{6}{5}-4\times\frac{2}{3}>0.$ 
			\item[(4.2.2)] $v$ has exactly one incident $4$-face.\\ 
			It follows from Lemma~\ref{forbid2}(4) that $v$ is a $(3,5,4,5,5)$-face. 
			Recall that $w(v \rightarrow f_1)  \leq  \frac{6}{5}$  by (R1.2) and (R7), 
			$w(v \rightarrow f_3)  \leq  1$ by (R3.2) and (R7), 
			and $w(v \rightarrow f_i)  \leq  \frac{2}{3}$ for each remaining $f_i$ by (R4.2) and (R7). 
			If $f_3$ is rich, then $w(v \rightarrow f_3)  \leq  \frac{2}{3}$ by (R3.2) and (R7). 
			Thus $\mu^*(v) \geq \mu(v) -\frac{6}{5}-4\times\frac{2}{3}>0.$ 
			If $f_3$ is not rich, then $f_2$ and $f_4$ are rich by Corollary~\ref{2faces}. 
			Consequently, $w(v \rightarrow f_i)  \leq  \frac{1}{2}$ for $i=2$ or $4$ by (R4.2) and (R7). 
		    Thus  $\mu^*(v) \geq \mu(v)-\frac{6}{5}-1-2\times \frac{1}{2}-\frac{2}{3}>0.$  
			\item[(4.2.3)] $v$ has exactly two incident $4$-faces.\\ 
			It follows from Lemma~\ref{forbid2}(3) and (4) that $v$ is a $(3,4,5,5,4)$- or a $(3,5,4,4,5)$-face. 
			Moreover, $v$ has at least three incident rich faces by Corollary~\ref{5vertex3rich}.
			Consequently, we have (i) $f_1$ and at least one $4$-face $f_i$ are rich, 
			(ii) $f_1$ and two $5^+$-faces are rich,  (iii) a $4$-face and two $5$-faces are rich, or
			(iv) two $4$-faces and a $5$-face are rich.\\ 
			Recall that $w(v \rightarrow f_1)  \leq  \frac{6}{5}$  by (R1.2) and (R7), 
			$w(v \rightarrow f_i)  \leq  1$ for each $4$-face $f_i$ by (R3.2) and (R7), 
			and $w(v \rightarrow f_i)  \leq  \frac{2}{3}$ for each $5$-face $f_i$ by (R4.2) and (R7).
			Additionally, $w(v \rightarrow f_1)  \leq  1$  if $f_1$ is rich by (R1.2) and (R7), 
			$w(v \rightarrow f_i)  \leq  \frac{2}{3}$ for each rich $4$-face $f_i$ by (R3.2) and (R7), 
			and $w(v \rightarrow f_i)  \leq  \frac{1}{2}$ for each  rich $5$-face $f_i$ by (R4.2) and (R7). \\ 
			If $f_1$ and a $4$-face $f_i$ are rich,
			then  $\mu^*(v) \geq \mu(v) -2\times 1 -3\times\frac{2}{3}=0.$  
			If $f_1$ and two $5^+$-faces are rich, 
			then  $\mu^*(v) \geq \mu(v) -1-2\times1 -2\times \frac{1}{2} =0.$  
			If a $4$-face and two $5^+$-faces are rich, 
			then $\mu^*(v) \geq \mu(v) -\frac{6}{5}-1-\frac{2}{3} - 2\times \frac{1}{2} >0.$ 
			If two $4$-faces and a $5$-face are rich, 
			then $\mu^*(v) \geq \mu(v) -\frac{6}{5}-3\times\frac{2}{3} -\frac{1}{2}>0.$  
			\end{itemize}
	\item[(4.3)] $v$ is incident to exactly two $3$-faces, say $f_1$ and $f_3.$\\ 
	It follows from Lemma~\ref{forbid2}(3) and (4) that $v$ has no incident $4$-faces. 
	This implies $v$ is a $(3,5,3,5,5)$-vertex. 	
	Recall that $w(v \rightarrow f_i)  \leq  \frac{6}{5}$  for each $3$-face $f_i$ by (R1.2) and (R7), 
	and $w(v \rightarrow f_i)  \leq  \frac{2}{3}$ for each $5$-face $f_i$ by (R4.2) and (R7).
	Additionally, $w(v \rightarrow f_1)  \leq  1$  if $f_1$ is rich by (R1.2) and (R7), 
	and $w(v \rightarrow f_i)  \leq  \frac{1}{2}$ for each  rich $5$-face $f_i$ by (R4.2) and (R7).\\ 
	If each incident $5$-face is rich, 
	then $\mu^*(v) \geq \mu(v) -2\times\frac{6}{5}-3\times\frac{1}{2}>0.$ 
	If $f_2$ is not rich, then $f_1$ and $f_3$ are rich by Corollary~\ref{2faces}. 
	Consequently, $f_4$ and $f_5$ are also rich. 
	Thus $\mu^*(v) \geq \mu(v) -3\times1-2\times\frac{1}{2}=0.$ 
	If $f_4$ is not rich, then  $f_3$ and $f_5$ are rich by Corollary~\ref{2faces}. 
	Consequently, $f_2$ is also rich. 
	Thus  $\mu^*(v) \geq \mu(v) -\frac{6}{5}-1-\frac{2}{3}-2\times\frac{1}{2}>0.$ 
	The case that $f_5$ is not rich is similar. 
	\end{itemize} 
\item[(5)]  A $6$-vertex $v$ is incident to a $3$-face that is adjacent to another $3$-face. 
	\begin{itemize} 
	\item[(5.1)] $v$ is incident to at least two consecutive $3$-faces.\\ 
	Let $f_1,\ldots, f_k$ be consecutive $3$-faces. 
	Similar to Case (2.1), we have $k \leq 3.$ 
	It follows from Lemma~\ref{forbid2}(1) and (2) that $v$ is a $(3,3,6^+,k_4,k_5,6^+)$-
	or $(3,3,3,6^+,k_5,6^+)$-face. 
	Since $w(v \rightarrow f_i)  \leq  \frac{3}{2}$ for each $5^-$-face $f_i$ by (R2.3), (R3.2), (R4.2), and (R7), 
	Thus $\mu^*(v) \geq \mu(v) - 4\times\frac{3}{2} = 0.$ 
	\item[(5.2)] $v$ has no adjacent incident  $3$-faces.\\ 
	Let $f_1$ be a $3$-face adjacent to another $3$-face. 
	It follows from Lemma~\ref{forbid2}(1) and (2) that $f_2$ and $f_6$ are $6^+$-faces. 
	Similar to Case 5.1, we obtain that $\mu^*(v) \geq \mu(v) - 4\times\frac{3}{2} = 0.$ 
	\end{itemize}
\item[(6)] A $6$-vertex $v$ is not incident to a $3$-face that is adjacent to another $3$-face.\\
	Consequently, $v$ is incident to at most three $3$-faces.  
	\begin{itemize}
	\item[(6.1)] $v$ is incident to at least one $6^+$-face.\\ 
	Recall that  $w(v \rightarrow f_i)  \leq  \frac{6}{5}$ for each $3$-face $f_i$ by (R1.2) and (R7), 
		and 	$w(v \rightarrow f_i)  \leq  \frac{3}{2}$ for each $k$-face $f_i$ where $k=4$ or $5$ 
		by (R3.2), (R4.2), and (R7). 
		Thus  $\mu^*(v) \geq \mu(v) - t\times\frac{6}{5}-(5-t) \times 1 >0$ where $t \leq 3$ is 
		the number of incident $3$-faces. 
	\item[(6.2)] $v$ has no incident $6^+$-face.
		\begin{itemize}
		\item[(6.2.1)] $v$ has no incident $3$-faces.\\
		By (R3.2), (R4.2), and (R7), we have $\mu^*(v) \geq \mu(v) - 6\times1= 0.$ 
		\item[(6.2.2)] $v$ has exactly one incident $3$-face, say $f_1.$\\ 
		It follows from Lemma~\ref{forbid2}(3) that $v$ is not a $(3,4,4,4,4,4)$-vertex. 
		Consequently, $v$ has $s$ $5$-faces where $t\geq 1.$ 
		Note that each incident face of $v$ is adjacent to another $4^+$-face.  
		It follows that $w(v \rightarrow f_i)  \leq  \frac{2}{3}$ for each $5$-face $f_i$ by (R4.2) and (R7). 
		Recall that  $w(v \rightarrow f_1)  \leq  \frac{6}{5}$  by (R1.2) and (R7), 
		and 	$w(v \rightarrow f_i)  \leq  1$ for each $4$-face $f.$ 
		Thus $\mu^*(v) \geq \mu(v) - \frac{6}{5}- s \times \frac{2}{3}-(5-s)\times 1  >0.$ 
		\item[(6.2.3)] $v$ has exactly two incident $3$-faces.\\
		Consequently, $v$ is a $(3,k_2,3,k_4,k_5,k_6)$- or $(3,k_2,k_3,3,k_5,k_6)$-vertex.\\ 
		Assume $v$ is a $(3,k_2,3,k_4,k_5,k_6)$-face. 
		Then $k_2=5$  by Lemma~\ref{forbid2}(3). 
		This implies $k_4=k_6=5$  by Lemma~\ref{forbid2}(4). 
		Since $v$ is a $(3,5,3,5,4^+,5)$-vertex,  
		we have $w(v \rightarrow f_i)  \leq  \frac{6}{5}$ for $i=1$ and $3$ by (R1.2) and (R7), 
		$w(v \rightarrow f_i)\leq 1$ for $i=2$ and $5$ by  (R3.2),(R4.2) and (R7), 
		and 	$w(v \rightarrow f_i)\leq \frac{2}{3}$ for $i=4$ and $6$ by  (R4.2) and (R7). 
		Thus $\mu^*(v) \geq \mu(v) - 2\times\frac{6}{5}-2\times1 - 2\times\frac{2}{3} >0.$\\ 
		Assume $v$ is a $(3,k_2,k_3,3,k_5,k_6)$-vertex. 
		It follows from Lemma~\ref{forbid2}(4) that $\{k_2,k_6\} \neq \{4,5\}.$ 
		If $k_2=k_6=4,$ then $k_3=k_5=5$ by   Lemma~\ref{forbid2}(3). 
		Consequently, we may assume that $v$ is a $(3,4,5,3,5,4)$- and $(3,5,5,3,5,5)$-vertex. 
		Recall that $w(v \rightarrow f_i)  \leq  \frac{6}{5}$ for $i=1$ and $4$ by (R1.2) and (R7),  
		$w(v \rightarrow f_i)\leq 1$ for each $4$-face $f_i$ by  (R3.2) and (R7). 
		and $w(v \rightarrow f_i)\leq \frac{2}{3}$ for each $5$-face $f_i$ by  (R4.2) and (R7). 
		Thus  a $(3,4,5,3,5,4)$-vertex has 
		$\mu^*(v) \geq \mu(v) - 2\times\frac{6}{5}-2\times 1 - 2\times\frac{2}{3} >0,$
		and a  $(3,5,5,3,5,5)$-vertex has  
		$\mu^*(v) \geq \mu(v) - 2\times\frac{6}{5}-4\times\frac{2}{3}-1 >0.$ 
		\item[(6.2.4)] $v$ has exactly three incident $3$-faces.\\ 
		Consequently, $v$ is a $(3,5,3,5,3,5)$-vertex  by Lemma~\ref{forbid2}(3). 
		It follows from Corollary~\ref{C(5,3,5,3)} that 
		two incident $5$-faces of $v$, say $f_2,$ and $f_4,$  are rich. 
		This implies $w(v \rightarrow f_i)  \leq  \frac{6}{5}$ for $i=1,3$ and $5$ by (R1.2) and (R7), 
		$w(v \rightarrow f_i)\leq \frac{1}{2}$ for $i=2$ and $4,$ 
		$w(v \rightarrow f_6)  \leq  1$ by  (R4.2) and (R7). 
		Thus  $\mu^*(v) \geq \mu(v) - 3\times\frac{6}{5}-2\times\frac{1}{2}-1 >0.$
		\end{itemize}
	\end{itemize}
\end{itemize}
\begin{itemize} 
\item[(7)] $v$ is a $k$-vertex where $k\geq 7.$ 
	\begin{itemize}
	\item[(7.1)] A vertex $v$ is  incident to a $3$-face that is adjacent to another $3$-face.\\
	Then  $v$ is incident to at least two $6^+$-faces by Lemma~\ref{forbid2}(1) and (2). 
%	**** (Make Lemma) ****
	Thus $\mu^*(v) \geq \mu(v) - (k-2)\times\frac{3}{2}>0$ by (R2.3), (R3.2), (R4.2), and (R7).  
	\item[(7.2)] A vertex $v$ is not incident to a $3$-face that is adjacent to another $3$-face.\\ 
	Consequently $v$ is incident to $t$ $3$-faces where $t \leq k/2.$ 
	Thus $\mu^*(v) \geq \mu(v) -t\times\frac{6}{5} - (k-t) \times 1>0$ by (R1.2), (R3.2), (R4.2), and (R7).
	\end{itemize}

\item[(8)] An inner $3$-face $f$ is not adjacent to another $3$-face.\\
	If $f$ has no incident flaw $4$-vertices,  
	then $\mu^*(f) \geq \mu(f) +3\times1=0$ by   (R1.1) adn (R1.2).
	If $f$ has an incident flaw vertex,   
	then $f$ is a $(4,4,5^+)$-face by Corollary~\ref{flaw}(2). 
	Recall that $w(v \rightarrow f)\geq \frac{9}{10}$ for an incident $4$-vertex $v$ by (R1.1), 
	and $w(v \rightarrow f)\geq \frac{6}{5}$ for an incident $5^+$-vertex $v$   by   (R1.2). 
	Thus $\mu^*(f) \geq \mu(f) +2\times\frac{9}{10}+\frac{6}{5}=0.$   

\item[(9)] An inner $3$-face $f$ is adjacent to another $3$-face.\\ 
	Note that we use only (R2) to calculate a new charge.  
	\begin{itemize}
	\item[(9.1)] A face $f$ is not in a trio.\\
	Then $\mu^*(f) \geq \mu(f) +3\times1=0.$  
	\item[(9.2)] A face $f$ is in a trio $T$ but not in $W_5$ formed by four inner $3$-faces. \\ 
%	(****trio***** Define by inner face ****)
	Let $f_1, f_2,$ and $f_3$ be $3$-faces in the same trio $T.$ 
	Define $\mu(T) := \mu(f_1)+\mu (f_2)+ \mu(f_3) = -9$ and 
	$\mu^*(T) := \mu^*(f_1)+\mu^* (f_2)+ \mu^*(f_3).$ 
	By (R8), it suffices to prove that   $\mu^*(T) \geq 0.$ 
	\begin{itemize}
		\item[(9.2.1)] A worst vertex is a $5^+$-vertex.\\  
		Then  $\mu^*(T) \geq \mu(T) +9\times1 = 0.$   
		\item[(9.2.2)] A worst vertex is a $4$-vertex and each worse vertex is a $4$-vertex.\\ 
		Then two bad vertices are  $5^+$-vertices by Corollary~\ref{2faces}. 
		Thus  $\mu^*(T) \geq \mu(T)+3\times\frac{2}{3}+2\times\frac{3}{2}+4\times1 = 0.$ 
		\item[(9.2.3)] A worst vertex is a $4$-vertex and one of worse vertices is a $5$-vertex.\\  
		Then   Corollary~\ref{2faces} yields that the other worse vertex 
		or at least one bad vertex is a $5^+$-vertex.
		Thus $\mu^*(T) \geq \mu(T) +3\times\frac{2}{3}+4\times\frac{5}{4}+2\times1  = 0$ or 
		$\mu^*(T) \geq \mu(T) +3\times\frac{2}{3}+2\times\frac{5}{4}
		+\frac{3}{2}+3\times1 = 0$, respectively. 
		\item[(9.2.3)] A worst vertex is a $4$-vertex and one of worse vertices is a $6^+$-vertex.\\  
		Then $\mu^*(T) \geq \mu(T) +3\times\frac{2}{3}+2\times\frac{3}{2}+4\times1 = 0.$
		\end{itemize}
	\item[(9.3)] A face $f$ is in $W_5$ formed by four inner $3$-faces incident to $v.$\\ 
	Let $f_1, f_2,f_3,$ and $f_4$ be $3$-faces in the same $W_5.$  
	Define $\mu(W_5) := \mu(f_1)+\mu (f_2)+ \mu(f_3) + \mu(f_4) = -12$ and 
	$\mu^*(W_5) := \mu^*(f_1)+\mu^* (f_2)+ \mu^*(f_3)+ \mu^*(f_4).$  
	By (R8), it suffices to prove that   $\mu^*(W_5) \geq 0.$  
	Note that each $3$-face in $W_5$ is adjacent to a $7^+$-face by Lemma~\ref{forbid2}(5). 
	Thus $W_5$ always obtains $4 \times \frac{1}{8}$ from four $7^+$-faces by (R5). 
% 	Note that we use only (R2) for calculating charges that $W_5$ is obtained from vertices.  
		\begin{itemize} 
		\item[(9.3.1)] Each vertex of $W_5$ is  a $5^-$-vertex.\\ 
		Then at least three of them are $5$-vertices by Corollary~\ref{W5}. 
		Thus  $\mu^*(W_5) \geq 
		\mu(W_5) +6\times\frac{5}{4}+2\times1+4\times\frac{1}{2}+4\times\frac{1}{8}= 0.$ 
		\item[(9.3.2)] Exactly one vertex of $W_5$ is a $6^+$-vertex.\\  
		Then one of the remaining vertices is a $5^+$-vertex by Corollary~\ref{2faces}. 
		Thus  $\mu^*(W_5) =\mu(W_5)+2\times\frac{3}{2}+2\times\frac{5}{4}
		+4\times1+4\times\frac{1}{2}+4\times\frac{1}{8}=0.$ 
		\item[(9.3.3)] At least two vertices of $W_5$ are  $6^+$-vertices.\\  
		Then $\mu^*(W_5) \geq \mu(W_5)
		+4\times\frac{3}{2}+4\times1+4\times\frac{1}{2}+4\times\frac{1}{8}> 0.$ 
		\end{itemize}
	\end{itemize}	
 
\item[(10)] $f$ is an inner $4$-face.\\ 
	We claim that $f$ is a $(4^+,4^+,4^+,5^+)$-face. 
	Suppose to the contrary that $f$ is a $(4,4,4,4)$-face. 
	By the minimality of  $G,$ 
	there is an $L$-coloring of $G -B(f)$ where $L$ is restricted to $G-B(f).$ 
	After the coloring, each vertex of $B(f)$ has at least two legal colors. 
	By Lemma~\ref{listcycle}, we can extend an $L$-coloring to $G,$ a contradiction.\\ 
	If $f$ is a $(4,4,4,5^+)$-face, then $\mu^*(f)\geq \mu(f) + 3\times\frac{1}{3}+1 =0$ by (R3).   
	If  $f$ is a $(4^+,4^+,5^+,5^+)$- or $(4^+,5^+,4^+,5^+)$-face, 
	then $f$ is a rich face and thus 
	$\mu^*(f) \geq \mu(f) +2\times\frac{1}{3}+2 \times \frac{2}{3} = 0$ by (R3).  
 
\item[(11)] $f$ is an inner $5$-face.
	\begin{itemize}
	\item[(11.1)] $f$ is a poor $5$-face, that is $f$ is a $(4,4,4,4,4)$-face.\\
	It follows from Lemma~\ref{forbid2}(2) that 
	each incident $4$-vertex of $f$ is incident to at most  two $3$-faces. \\
	If an incident vertex $v$ of $f$ is incident to at most one $3$-face, 
	then $w(v \rightarrow f) = \frac{1}{3}$ by (R4.1), 
	otherwise  $v$ is a flaw vertex by Lemma~\ref{forbid2}(2) and (3) 
	which implies  $w(v \rightarrow f) = \frac{1}{5}$ by (R4.1). 
	Consequently, $f$ gains at least $\frac{1}{5}$ from each incident vertex. 	
	Thus $\mu^*(f)\geq \mu(f) +5\times\frac{1}{5}  = 0.$ 
	\item[(11.2)] $f$ is a $(4,4,4,4,5^+)$-face.
		\begin{itemize} 
		\item[(11.2.1)]$f$ is adjacent to at least one $4^+$-face $g.$\\  
		It follows from (R4.2) that $w(v \rightarrow f) = \frac{2}{3}$ for an incident  $5^+$-vertex $v$ of $f.$ 
		Consider a $4$-vertex $u \in V(B(f)) \cap V(B(g)).$  
		It follows from Lemma~\ref{forbid2}(2) that  $u$ is incident to at most one $3$-face. 
		Consequently,  $w(u \rightarrow f) = \frac{1}{3}$ by (R4.1). 
	 	Thus $\mu^*(f)\geq \mu(f) +\frac{2}{3}+\frac{1}{3}  = 0.$  
		\item[(11.2.2)] $f$ is adjacent to five $3$-faces.\\ 
		Then  $\mu^*(f) = \mu(f) +1  = 0$ by (R4.2).   
		\end{itemize}
	\item[(11.3)] $f$ is  a rich face with $t$ incident $5^+$-vertices.\\ 
	Then  $\mu^*(f)\geq \mu(f)+ t\times\frac{1}{t} = 0$ by (R4.2).  
	\end{itemize}
\item[(12)] $f$ is an inner $6^+$-face.\\ 
	If $f$ is a $6$-face, then $\mu^*(f)=\mu(f) = 0.$   
	If $f$ is a $k$-face where $k \geq 7,$ 
	then $\mu^*(f)\geq \mu(f) - k\times \frac{1}{8} >0$ by  (R5). 
\item[(13)] $f$ is an extreme face. 
	\begin{itemize} 
	\item[(13.1)] $f$ is a $3$-face that shares exactly one vertex, say $u,$ with $C_0.$\\ 
	It follows from (R7) that $w(D \rightarrow f) = 2$ and 
	$w(v \rightarrow f) = \frac{1}{2}$ for each incident vertex $v$ in $int(C_0).$ 
	Thus $\mu^*(f) = \mu(f) +2+2\times\frac{1}{2}  = 0.$ 
	\item[(13.2)] $f$ is a $3$-face that shares an edge  with $C_0.$\\ 
	It follows from (R7) that $w(D \rightarrow f) = \frac{5}{2}$    
	and $w(v \rightarrow f) = \frac{1}{2}$ for an incident vertex $v$ in $int(C_0).$ 
	Thus $\mu^*(f) = \mu(f) +\frac{5}{2} +\frac{1}{2} = 0.$ 
	\item[(13.3)] $f$ is a $4$- or $5$-face.\\ 
	Then $\mu^*(f)\geq \mu(f) +2 \geq 0$ by (R7).  
	\end{itemize}
	
\item[(14)] $D$ is the  unbounded face.\\
Let $f_3', f'$ be the number of $3$-faces sharing exactly one edge with $D$, 
$3$-faces sharing exactly one vertex with $D$ or $4$-or $5$-faces sharing vertices with $D$, respectively. 
Let $E(C_0, V(G)-C_0)$ be the set of edges between $C_0$ and $V(G)-C_0$ 
and let $e(C_0, V(G)-C_0)$ be its size.  Then by (R6) and (R7),
	\begin{align}
	\mu^*(D)&=3+6+\sum_{v\in C_0} (2d(v)-6)-\frac{12}{5}f_3'-2f'\\
	&=9+2\sum_{v\in C_0} (d(v)-2)-2\times3-\frac{12}{5}f_3'-2f'\\
	&=3-\frac{2}{5}f_3'+2(e(C_0,V(G)-C_0)-f_3'-f')
	\end{align}
So we may consider that each edge $e\in E(C_0,V(G)-C_0)$ gives a charge of $2$ to $D.$  
Since each $5^-$-face is a cycle, it contains two edges in $E(C_0,V(G)-C_0)$. 
It follows  that $e(C_0,V(G)-C_0)-f_3'-f'\ge0$. Note that $f_3'\le3$. Thus $\mu^*(D)>0.$ 
\end{itemize}
%\indent This completes the proof. 

% \section{Acknowledgments}
% \indent The first author is supported by Development and Promotion 
% of Science and Technology Talents Project (DPST).

\end{document}